\documentclass[11 pt]{amsart}
\usepackage{amsmath, amsthm,amssymb,bbm,enumerate,mathrsfs,mathtools}
\usepackage{a4wide}
\usepackage{tikz,verbatim,xcolor,graphicx}
\usepackage[linktocpage=true]{hyperref}
\usepackage{makecell}
\usepackage[margin=3cm]{geometry} 
\usepackage{verbatim,bbm,theoremref}

\renewcommand\labelenumi{(\roman{enumi})}
\renewcommand\theenumi\labelenumi
\newcommand{\eps}{\ensuremath{\varepsilon}}
\newcommand{\ind}{\mathbbm{1}}

\newcommand{\remove}[1]{}

\newcommand{\E}{\mathbbm{E}}
\newcommand{\F}{\mathbbm{F}}
\newcommand{\C}{\mathbbm{C}}

\newcommand{\N}{\mathbbm{N}}
\newcommand{\cQ}{\mathcal{Q}}
\newcommand{\R}{\mathbbm{R}}

\newcommand{\Cc}{\mathcal{C}}

\newcommand{\fdag}{\ensuremath{f^{\dagger}}}
\newcommand{\gownn}[1]{G_n^{\alpha}[#1]}
\newcommand{\Z}{\mathbbm{Z}}
\newcommand{\fint}{\left[ -\frac 12, \frac 12 \right]}
\newcommand{\com}{\Delta}

\newcommand{\Fqn}{\F_q^n}
\newcommand{\bx}{\mathbf{x} }

\newcommand{\sm}{\setminus}
\newcommand{\se}{\subseteq}

\newcommand{\xv}{\mathbf{x} }
\newcommand{\compl}[1]{{#1}^{\mathsf{c}}} 
\newcommand{\Tr}{\mathrm{Tr}}

\newcommand{\solA}[2]{\sol(#1;#2)} 

\title{\vspace{-0.8cm}On uncommon systems of equations}
\author{Nina Kam\v{c}ev}
\address{Department of Mathematics, Faculty of Science, University of Zagreb, Croatia}
\email{nina.kamcev@math.hr}

\author{Anita Liebenau}
\address{School of Mathematics and Statistics, UNSW Sydney, NSW 2052, Australia.}
\email{a.liebenau@unsw.edu.au}

\author{Natasha Morrison}
\address{Mathematics and Statistics, University of Victoria, Victoria, B.C. Canada V8P 5C2}
\email{nmorrison@uvic.ca}

\newtheoremstyle{case}{}{}{\normalfont}{}{\itshape}{:}{ }{}
 
\newtheorem{thm}{Theorem}
\newtheorem{lem}[thm]{Lemma}
\newtheorem{prop}[thm]{Proposition}
\newtheorem{conj}[thm]{Conjecture}
\newtheorem{cor}[thm]{Corollary}
\newtheorem{claim}[thm]{Claim}

\theoremstyle{definition}
\newtheorem{defn}[thm]{Definition}
\newtheorem{obs}[thm]{Observation}
\newtheorem{rem}[thm]{Remark}

\newtheorem{question}[thm]{Question}

\newtheoremstyle{case}{}{}{\normalfont}{}{\itshape}{\normalfont:}{ }{}

\theoremstyle{case}

\numberwithin{equation}{section}
\numberwithin{thm}{section}

\newcommand{\er}[1]{\E \left[ #1 \right]}
\newcommand{\erm}[2]{\E_{#1} \left[ #2 \right]}

 \newcommand{\errm}[2]{\E_{#1} \left[ #2 \right]}

\newcommand{\sol}{\mathrm{sol}}

\begin{document}

\maketitle

\begin{abstract}
A linear system $L$ over $\F_q$ is \emph{common} if the number of monochromatic solutions to $L=0$ in any two-colouring of $\F_q^n$ is asymptotically at least the expected number of monochromatic solutions in a random two-colouring of $\F_q^n$. Motivated by existing results for specific systems (such as Schur triples and arithmetic progressions), as well as extensive research on common and Sidorenko graphs, Saad and Wolf recently initiated the systematic study of common systems of linear equations. 

Building upon earlier work of Cameron, Cilleruelo and Serra, as well as Saad and Wolf, common linear equations have recently been fully characterised by Fox, Pham and Zhao, who asked about common \emph{systems} of equations. In this paper we move towards a classification of common systems of two or more linear equations. In particular we prove that any system containing an arithmetic progression of length four is uncommon, resolving a question of Saad and Wolf. This follows from a more general result which allows us to deduce the uncommonness of a general system from certain properties of one- or two-equation subsystems.
\end{abstract}

\maketitle

\section{Introduction}
Finding arithmetic structure in subsets of the integers is a fundamental theme throughout additive combinatorics. Results such as Szemer\'edi’s Theorem on arithmetic progressions in dense sets and the Green-Tao theorem on arithmetic progressions in the primes are cornerstone results of modern mathematics that have inspired and united mathematicians specialising in areas such as ergodic theory, number theory, Fourier analysis and combinatorics. Naturally, these density theorems go hand in hand with analogous results that guarantee monochromatic substructures in any colouring of a certain object. Classical examples are Van der Waerden's theorem and Schur's theorem that assert that any colouring of $[n]$ (for large $n$) contains a monochromatic arithmetic progression, or a monochromatic solution to $x+y = z$, called a \emph{Schur triple}, respectively.

Given these results, it is natural to wonder how many such monochromatic substructures are present. In fact, the proofs of Schur's theorem and Van der Waerden's theorem imply that any colouring of $[n]$ contains at least a positive proportion of Schur triples and arithmetic progressions, respectively. In 1996, Graham, R\"odl and Ruci\'nski~\cite{grr96} asked about the minimal proportion of monochromatic Schur triples in two-colourings of $[n]$, which was independently resolved in~\cite{datskovsky03,rz98,schoen99} and shown to be far below the number expected in a random colouring. In contrast to this, in the finite Abelian group setting, Cameron, Cilleruelo and Serra~\cite{ccs07} showed that the random colouring minimises the number of monochromatic solutions to any equation in an odd number of variables, which includes Schur triples. 

The phenomenon that a random colouring minimises the number of monochromatic copies of a particular substructure is present throughout combinatorics. In particular, a graph $H$ is called common if, for $n$ large, it has the property that number of monochromatic copies in any colouring of the edges of the $n$-vertex clique, denoted by $K_n$, is minimised by a random 2-colouring. The study of such graphs goes back to Erd\H{o}s~\cite{erdos62}, who conjectured in 1962 that $K_k$ is common for $k \ge 4$. An earlier result of Goodman~\cite{goodman59} states that $K_3$ is common, and in 1980, Burr and Rosta~\cite{br80} extended Erd\H{o}s' conjecture to arbitrary $H.$ In 1989, the Burr-Rosta conjecture was independently disproved by Thomason~\cite{thomason89}, who also disproved Erd\H{o}s' conjecture by showing that $K_4$ is uncommon, and by Sidorenko~\cite{sid89}. Although these conjectures are false, the desire to characterise common and uncommon graphs continues to this day, for example see~\cite{conlon12,com3,com1,com2,jst96,sid89,sid93,thomason97}. A related conjecture is the well known Sidorenko's conjecture on subgraph densities of bipartite graphs which has received considerable attention~\cite{sid1,sid2,sid3,sid4,sid5,sid6,sid7}. 

Inspired by the above-mentioned results for graphs and arithmetic structures, Saad and Wolf~\cite{sw17} initiated a systematic study of the number of monochromatic solutions to linear patterns in more generality. From now on, we work over the finite field $\F_q,$ where $q$ is a prime power, following e.g.~\cite{fpz19,sw17}.  Consider a linear map $L : (\Fqn)^k \to (\Fqn)^m$ with coefficients in~$\F_q$. Say that $L$ is \emph{common} if the density of monochromatic solutions to the system of linear equations $L(\xv)=0$ in any two-colouring of $\Fqn$ is asymptotically minimised by the expected density of solutions in a random two-colouring. 
We provide a more practical definition in Section~\ref{sec:prelims}, \thref{def:fctdef}.

Earlier work focused on systems consisting of a single equation $a_1x_1 + \dots + a_kx_k = 0$ with coefficients $a_i \in \F_q^\times = \F_q \setminus \{ 0\}$. 
As mentioned above, whenever $k$ is odd, such an equation is common~\cite{ccs07}. For even $k$, Saad and Wolf~\cite{sw17} proved that the equation is common whenever the coefficients can be partitioned into pairs, each pair summing to zero. They conjectured that this sufficient condition is also necessary, which was confirmed by Fox, Pham and Zhao~\cite{fpz19}. Hence homogeneous equations of this form are completely characterised. 

Fox, Pham and Zhao~\cite{fpz19} asked about a similar characterisation for common \emph{systems} of equations, hoping that it might lead to a better understanding of the analogous properties for graphs and hypergraphs, but they note that they do not have a guess for such a characterisation. A number of specific systems had been considered by Saad and Wolf, but no results covering a general class of systems are known. It would be desirable to have such a classification for common systems of two or more linear equations and we make significant progress towards this goal. As we will see, characterising the properties of systems of equations is much more complex than the single equation case. Our results focus on showing that, in many cases, the presence of a particular subsystem is sufficient for a system to be uncommon. 

In order to state our main results, we now introduce some definitions. Let $L$ be a collection of $m$ linear forms $L_1,\ldots,L_m$ in $k$ variables with coefficients in $\F_q.$ 
For an $\ell$-variable system $L',$ we say that {\em $L$ induces $L'$ as a subsystem} if there exists a subset $\{i_1,\ldots,i_{\ell}\} \se[k]$ such that for ${\bf x} \in \F_q^k$, $L(x_1, \dots, x_k)=0$ implies that $L'(x_{i_1},\ldots, x_{i_{\ell}})=0.$ 
We expand upon the notion of induced subsystems and its connection to  submatrices in Section~\ref{sec:prelims}, where we present a more algorithmic way to view this concept.
A system $L$ is called an {\em $(m\times k)$-system} if the rows of the coefficient matrix of $L$ are linearly independent over $\F_q$. 
Finally, following the terminology from~\cite{rr97}, $L$ is \emph{redundant} if it induces the equation $x_i - x_j = 0$, for some $i \not= j$, and \emph{irredundant} otherwise. 
We can now state our main result.

\begin{thm} \thlabel{t:4APs}
	Let $q$ be an odd prime power, let $2 \le m < k$ be integers and let $L$ be an irredundant  $(m \times k)$-system over $\mathbb{F}_q$. If $L$ induces a $(2 \times 4)$-system, then $L$ is uncommon. 
\end{thm}

We remark that an $(m \times k)$-system always satisfies $m \le k$ and that $m=k$ implies that $L$ is trivially common (see \thref{rem:mk}). Furthermore, the irredundancy condition on the subsystem (implied by the irredundancy of $L$) is required, as, for example, the system $x_1=x_2=\ldots = x_k$ is common. We restrict our attention to irredundant systems as, in Section~\ref{sec:prelims}, we see that the commonness of a redundant system is determined by the commonness of an irredundant subsystem. 

For any prime $p>3$ and any power $q$ of $p$, an arithmetic progression of length four (4-AP) is an irredundant $(2 \times 4)$-system over $\F_q$, and so in this case \thref{t:4APs} not only tells us that a 4-AP itself is uncommon, but it resolves a question of Saad and Wolf~\cite{sw17} that any (irredundant) system inducing a 4-AP is uncommon.  

\begin{cor}\thlabel{c:main}
	Let $q$ be an odd prime power, let $2 \le m < k$ be integers and let $L$ be an irredundant $(m \times k)$-system over $\F_q$. If $L$ induces an arithmetic progression of length four, then $L$ is uncommon. 
\end{cor}

Both \thref{t:4APs} and \thref{c:main} can be seen as analogues of the famous result of Jagger, \v{S}\v{t}ov\'{\i}\v{c}ek and Thomason~\cite{jst96}, showing that any graph containing a particular small subgraph, namely $K_4$, is uncommon. In our case we see that any system containing a small subsystem with particular structure is uncommon.

Since any arithmetic progression of length $k \ge 4$ induces a $4$-AP, our result implies that any $k$-AP, and any system inducing a $k$-AP is uncommon. This concludes earlier investigations into the properties of $k$-APs. Notably, a colouring by Wolf~\cite{w10}, based on a construction of Gowers~\cite{gowers20}, showed that 4-APs are uncommon over $\mathbb{Z}_N,$ for large $N$. 
\remove{This was based on a construction of  Gowers~\cite{gowers20}, who used techniques from quadratic Fourier analysis to find a `uniform' set of density $1/2$ that certified uncommonness in $\mathbb{Z}_N$. }
See also~\cite{sw17} for a proof of uncommoness of 4-APs over $\F_5.$ We remark that \thref{c:main} has been independently\footnote{We learnt of \cite{v21} in the final stages of preparing our original preprint.} proved by Versteegen~\cite{v21}, whose work also covers the more general case of finite abelian groups. 

These earlier results in~\cite{gowers20,sw17,w10} used geometric intuition relying on strong structural properties of arithmetic progressions, but perhaps surprisingly we do not utilise these properties here. Excitingly, \thref{t:4APs} applies to {\em any} irredundant $(2 \times 4)$-system (as we see below, irredundancy is not a strong condition). Not only does this determine the uncommonness of a large and general family of systems, but as we only use weak conditions on the structure of the solution space, there is reason to believe that our techniques could be used to characterise other general families of systems. 

\thref{t:4APs} is obtained as a consequence of a much more general (and more technical) result (Theorem~\ref{t:main-uncommon}) that can be applied to find a large class of uncommon systems. We postpone the statement of Theorem~\ref{t:main-uncommon} to Section~\ref{sec:main-thm}, by which time we will have introduced the required concepts. Roughly speaking, Theorem~\ref{t:main-uncommon} provides a sufficient condition for a system to be uncommon based solely on particular `critical' subsystems (which turn out to have rank at most two). So it provides a means of understanding a potentially complex high-rank system, by understanding certain low-rank subsystems. One exciting consequence of this is that it opens up avenues for using discrete Fourier analysis in studying systems with two or more equations. 

We now state a fairly straightforward application of Theorem~\ref{t:main-uncommon}. 
The \emph{length} of an equation $E$ is the number of variables in $E$ with \emph{non-zero} coefficients. Given a system $L$, let $s(L)$ denote the minimal length of an equation induced by $L$. 
 
\begin{thm} \thlabel{t:main-evens}
	Let $q$ be a prime power, let $2 \le m < k$ be integers and let $L$ be an $(m \times k)$-system over $\mathbb{F}_q$ such that $s(L)$ is even. If every equation of length $s(L)$ induced by $L$ is uncommon, then $L$ is uncommon. 
\end{thm}

As a single equation $E$ with even support is only common if its coefficients can be partitioned into pairs, each summing to zero, in some sense, a `{typical}' equation of even length is uncommon. Similarly, Theorem~\ref{t:main-evens} says that a {`typical'} system with $s(L)$ even is uncommon. The hypothesis that $s(L)$ is even is more than an artefact of our proofs, which will be evident from the arguments, and the hypothesis is implicitly present in the results of~\cite{ccs07,fpz19,gowers20}.

In Section~\ref{sec:prelims} we introduce the `functional' notion of commonness that will be used throughout the paper. We also briefly recall the definitions from discrete Fourier analysis that will be used and introduce the notion of critical sets and subsystems, which is vital for the formulation of Theorem~\ref{t:main-uncommon} (our main technical theorem). In Section~\ref{sec:main-thm} we prove Theorem~\ref{t:main-uncommon} and deduce Theorem~\ref{t:main-evens}. Then, in Section~\ref{sec:unc-4-var} we prove that any irredundant $(2 \times 4)$-system is uncommon. In fact, our main result from this section (\thref{l:psi-neg}) plays a fundamental part in the proof of \thref{t:4APs}, which is given in Section~\ref{sec:ProofOfMain1}. We conclude with some interesting open questions and directions for future research in Section~\ref{sec:rems}.

\section{Preliminaries}\label{sec:prelims}

Our first goal is to formulate a `functional' notion of commonness, which will turn out to be more convenient for our purposes. 

Throughout the section, let $q$ be a prime power and let $n \ge 1$. 
Let $L$ be an $(m \times k)$-system, consisting of the linear forms $L_1,\ldots,L_m$ with coefficients in $\F_q.$ We identify $L$ with the $(m\times k)$-matrix whose rows consist of the coefficients of $L_1,\ldots, L_m.$ Then the solution set of $L$ in $A \subseteq \Fqn$ is
$$\solA{L}{A} := \{\bx = (x_1, \dots, x_k) \in A^{k}: L \bx^T = 0 \}.$$ 
Clearly, the set $\solA{L}{\Fqn}$ is invariant under row operations of the matrix $L.$ Thus, commonness does not depend on the choice of the representative matrix $L$. We will work interchangeably with systems of linear forms and their corresponding matrices throughout the paper. We  also write $L(\bx)$ or $L(x_1,\ldots,x_k)$ when we want to specify the variables.

Note that every solution $\xv\in\solA{L}{A}$ corresponds naturally to an $(n\times k)$-array in which the columns are elements in $A$ and every row is a solution to $L=0$ over $\F_q.$ The following is then immediate by recalling that the row vectors of an $(m \times k)$-system $L$ are linearly independent. 
\begin{obs}\thlabel{ob:n-vec}
	Let $L$ be an $(m\times k)$-system. 
	Then $|\solA{L}{\F_q^n}| = |\solA{L}{\F_q}|^n = q^{n(k-m)}$. 
\end{obs}

For a function $f: \Fqn \to \R$ define the {\em density of solutions} to a system $L(x_1,\ldots,x_k)$ with respect to $f$ to be
\begin{equation}\label{eq:lamf}
\Lambda_{L} (f) := \frac{1}{|\solA{L}{\Fqn}|}\sum_{\bx \in \solA{L}{\Fqn}} f(x_1)f(x_2) \dots f(x_k), 
\end{equation}
and let $\com_L(f) := \Lambda_L(1/2 + f) + \Lambda_L(1/2 - f).$
Throughout the paper we will work with the following `functional' definition of commonness,  
which we will see is equivalent to the version stated in the introduction. 
\begin{prop}\thlabel{def:fctdef}
Let $L$ be an irredundant $k$-variable system over $\mathbb{F}_q.$ Then $L$ is {\em common} if and only if for every $n$ and every $f: \Fqn \to [-1/2, 1/2]$, we have
$\com_L(f) 
\geq 2^{1-k}.$
\end{prop}

Let us briefly explain why \thref{def:fctdef} holds. 
Let $L$ be an irredundant system such that for every $n$ and every $f: \Fqn \to [-1/2, 1/2]$, we have
$\com_L(f) \geq 2^{1-k}.$ Then for every $n$ and every $A \subseteq \Fqn$ we have 
\begin{equation}\label{eq:commonSet} 
|\sol(L;A)| + |\solA{L}{\compl{A}}| \ge 2^{1-k}|\solA{L}{\Fqn}|,\end{equation} 
where $\compl{A}=\Fqn\sm A.$ This can be easily seen by taking $f=\ind_A-1/2.$ 
That is, $L$ is common according to the definition given in the introduction. 
\thref{def:fctdef} states that, in fact, this set-theoretic definition of a system to be common is equivalent to the functional definition we give above. The argument for a 1-equation system can be found in~\cite{fpz19} which translates directly to our setting as long as $L$ is irredundant. 
We observe that the left-hand side of~\eqref{eq:commonSet} is simply the number of monochromatic solutions in a 2-colouring of $\Fqn$ given by the partition $(A,\compl{A})$. Now let $P$ be a random subset of $\Fqn$ that contains every $y\in \Fqn$ with probability $1/2.$ Then the expected number of monochromatic solutions is equal to $(2^{1-k}-o_n(1))|\solA{L}{\Fqn}|$, where $o_n(1)$  accounts for the proportion of $\xv\in (\Fqn)^k$  with $L(\xv) = 0$ and not all $k$ coordinates distinct, which goes to 0 as $n\to\infty$ for any irredundant system $L.$ 

We now briefly remark why it suffices to restrict our attention to the consideration of irredundant systems.  Let $L(x_1,\ldots,x_k)$ be an irredundant (that is, it does not induce the equation $x_i - x_j = 0$, for some $i \neq j$) system and let $L'(x_1,\ldots,x_{k+1})$ be the system obtained from $L$ by including the form $x_k - x_{k+1}$. For $n\ge 1,$
let $A \subseteq \Fqn$ and observe that there is a clear one-to-one correspondence between the solutions to $L=0$ in $A$ and the solutions to $L'=0$ in $A$. 
We deduce that $L$ is common if and only if $L'$ is. Note in light of \thref{def:fctdef}, that for a redundant $k$-variable system, the benchmark for commonness is no longer $2^{1-k}.$

Let us briefly comment on the case when $m=k.$ In this case, the only solution to $L=0$ is the all zero solution (since we assume linear independence of the rows), so $\solA{L}{A} +\solA{L}{\compl{A}}= 1$ for any $A\se \Fqn,$ and $|\solA{L}{\Fqn}|=1.$ As a result we obtain the following. 
\begin{obs}\thlabel{rem:mk}
	Let $L$ be a $(k \times k)$-system where $k \ge 2$. 
	Then $L$ is common.  
\end{obs}
Thus when characterising conditions that force uncommonness in systems of multiple equations, it suffices to consider $(m \times k)$-systems, for $2 \le m < k$.

\subsection{Discrete Fourier Analysis}

Here we introduce some concepts from discrete Fourier analysis that will be used later (in Subsection~\ref{subsec:even} and Section~\ref{sec:ProofOfMain1}). We remark that the following definitions can be generalised to any finite abelian group, but as we are only concerned $\F_q^n$ here, we state them in this setting. More details may be found in~\cite{tv2006}.
 
The set of homomorphisms from $\F_q^n$ to $\C^{\times}$ forms a group and is denoted by $\widehat{\F_q^n}$. Now,  $\Fqn$ and $\widehat{\Fqn}$ can be seen to be isomorphic by identifying each $r \in \F_q^n$ with the \emph{character} $\chi_r$ that maps $x$ to $\omega^{\Tr( r \cdot x )},$ where $\omega = \exp(2\pi i/p),$ $p$ is the characteristic of $\F_q,$ $\Tr:\F_q\to\F_p$ is the standard trace map (which is non-degenerate and linear as a map between vector spaces over $\F_p$), and where $r \cdot x$ denotes the standard dot product $\F_q^n\times \F_q^n\to \F_q$. See for example~\cite{finite-fields-book} for more details. 
When there is no danger of confusion, $r$ is used to denote $\chi_r$. For a function $f: \F_q^n \rightarrow \C$, the \emph{Fourier transform} of $f$ is the function $\widehat{f}: \widehat{\F_q^n} \rightarrow \C$, defined by
	$$\widehat{f}(r) := \E_{x \in \F_q^n}\; f(x)\:\overline{r(x)} = \E_{x \in \F_q^n}\; f(x)\:\omega^{- \Tr( r\cdot x)},$$
where we use $\E_{x\in A} g(x)=\sum_{x\in A} g(x)/|A|$ to denote \footnote{We remark that in Lemma~\ref{l:random-fcs}, $\E$ is used to denote the expectation of a non-uniform random variable, but this is the only use of this notation in this way.} 
the average of $g$ over all $x\in A$.
The values $\widehat{f}(r)$ are called the \emph{Fourier coefficients} of $f$.

Let $L(x_1,\ldots,x_k)$ be the single equation system $a_1x_1 + \dots + a_kx_k = 0$ with coefficients $a_i \in \F_q \setminus \{ 0\}$ and let $f:\F_q^n \rightarrow \R$. A straightforward application of the definitions gives the standard equality
\begin{equation}\label{eq:single-eq}
	\Lambda_L(f) = \sum_{r \in \widehat{\F_q^n}}\widehat{f}(a_1r)\ldots\widehat{f}(a_k r),	
\end{equation}
see also~\cite{fpz19}. This is one example of how Fourier analysis can yield powerful tools for dealing with linear systems consisting of a single equation. Indeed, this relationship is strongly utilised in \cite{fpz19}, where the authors find a function $f$ with suitable Fourier  coefficients to show uncommonness of a single equation.

\subsection{Critical sets and subsystems}

Our main theorem (Theorem~\ref{t:main-uncommon}), which is applied to prove both Theorem~\ref{t:4APs} and Theorem~\ref{t:main-evens}, relates the problem of determining whether a system is uncommon to that of understanding the uncommonness of particular subsystems corresponding to \emph{critical sets}. The aim of this section is to motivate and introduce these sets and subsystems, before we formally state and prove Theorem~\ref{t:main-uncommon} in the next section.

Let $q$ be a prime power, let $m < k$ be integers and let $L$ be an $(m \times k)$-system over $\F_q$. 
For a set $B=\{i_1,\ldots,i_{\ell}\}\se [k]$ and an $\ell$-variable system $L',$ we say that {\em $L$ induces $L'$ on $B$} if, for every ${\bf x} \in \F_q^k$, $L(x_1, \dots, x_k)=0$ implies that $L'(x_{i_1},\ldots, x_{i_{\ell}})=0.$ Thus, $L$ induces $L'$ as a subsystem (as defined in the introduction) if it induces $L'$ on $B$ for some $B\se[k].$

Let us briefly comment on an equivalent viewpoint on induced subsystems via  coefficient matrices. Suppose that an $(m\times k)$-system $L$ induces an $(m' \times \ell)$-system $L'$ on a set $B,$ and let $M'$ be a matrix representing $L'$. We claim that then there is a matrix $M$ representing $L$ such that $M'$ is a submatrix of $M$ on the columns corresponding to $B$, and on rows $[m']$. Moreover, $M_{ij} = 0$ whenever $i \in [m']$ and $j \notin B.$ To see this, note first that $M$ is a matrix representing $L$ whenever the rows of $M$ form a basis for the orthogonal complement of $\sol(L;\F_q)=\ker(M)$ in $\F_q^k$. Similarly, the rows of $M'$ are a basis for the orthogonal complement of $\ker(M')$ in $\F_q^\ell,$ since $M'$ represents $L'.$
 Now form a matrix $\widetilde{M'}$ by adding $k - \ell$ all-zero columns to the right of $M'$. The crucial observation is that the rows of $\widetilde{M'}$ are in the orthogonal complement of $\sol(L)$. Indeed, any $(x_1\ldots,x_k)\in \sol(L)$ satisfies $L'(x_{i_1}\ldots,x_{i_\ell}) = 0$ since $L$ induces $L',$ which then implies that $\widetilde M' (x_1\ldots,x_k)^\top=0.$ 
 Therefore, the row vectors of $\widetilde{M'}$ can be extended to a basis of $\sol(L;\F_q)^\perp,$ which can then be taken as the row vectors of a matrix $M$ which has the required form. 

Recall that $s(L)$ denotes the minimal length of an equation induced by $L$. We remark that $s(L) \leq s(L')$ whenever $L$ induces $L'.$ As mentioned above, the parity of $s(L)$ plays an important part in our proofs. With this in mind, define 
$$	c(L) :=
\begin{cases}
	s(L) &\text{if $s(L)$ is even,}\\
	s(L) + 1 &\text{if $s(L)$ is odd.}
\end{cases}$$
Call a set $B \subseteq [k]$ \emph{critical for $L$} if $|B| = c(L)$ and there exists a system $L '$ that is induced by $L$ on $B$. 
Let $\mathcal{C}(L)$ denote the family of sets that are critical for $L$. 
Call a system $L'$ \emph{critical for L} (or simply \emph{critical}) when it is  induced by $L$ on a critical set. Note that $L$ may itself be critical (for example, if it is a single equation of even length) or $L$ may have no critical sets (for example, if $L$ is a single equation of odd length). 
For $B\in\Cc(L),$ let $m_B$ be the maximal $m'$ such that there exists an $(m'\times b)$-system $L'$ such that $L$ induces $L'$ on $B,$ where $b=|B|.$ 

We will now state some fundamental properties of critical sets and systems. We call two $k$-variable systems $L$ and $L'$ equivalent if $L$ induces $L'$  and $L'$ induces $L.$ From the perspective of matrices over $\Fqn,$ the systems $L$ and $L '$ are equivalent if one can be obtained from the other by a sequence of elementary row operations. The proof of the next lemma will be given after we discuss some important consequences. We include the assumption $m\ge 2$ merely to avoid technicalities when $L$ is an odd equation. (Recall that single-equation systems are already fully characterised in terms of commonness.) 
 
\begin{lem}\thlabel{lem:useful}
 	Let $q$ be a prime power, let $2\le m \le k$ be integers and let $L$ be an $(m \times k)$-system over $\F_q$. Let $B \in \mathcal{C}(L)$. Then the following hold.
 	\begin{enumerate}[(i)]
 		\item The $(m_B\times c(L))$-system induced by $L$ on $B$ is unique (up to equivalence). 
 		\item If $s(L)$ is even, then $m_B = 1.$ 
 		\item If $s(L)$ is odd, then $m_B\in\{1,2\}.$ 
 	\end{enumerate}
 \end{lem}

Item (i) allows us to make the following definition. 
\begin{defn}[The critical system $L_B$]
Given an $(m \times k)$-system $L$ over $\F_q$ and a critical set $B\se\Cc(L),$ define $L_B$ to be the $(m_B\times k)$-system that is induced by $L$ on $B$. 
\end{defn}

\thref{lem:useful} implies that for any critical $B$, the critical system $L_B$ is of rank one or two. 
The key property about the systems $L_B$ is expressed in the following lemma which will enable us to convert the problem of showing that a system is uncommon to showing that the critical subsystems satisfy particular properties. 

\begin{lem}\thlabel{lem:useful4}
Let $q,$ $m$, $k$, $L$ and $B$ be as in \thref{lem:useful}. Any solution to $L_B$ in $\F_q^n$ extends to $q^{n(k - c(L) - m + m_B)}$ solutions of $L$.  
\end{lem}

Both lemmas follow from elementary linear algebra. 

\begin{proof}[Proof of \thref{lem:useful}]
	For (i), suppose for contradiction that there are two non-equivalent $(m_B\times \ell)$-systems $L_1$ and $L_2$ induced by $L$ on $B$, where $\ell=c(L)$. Let $v_1,\ldots,v_{m_B}$ be the rows of the matrix $L_1$ and let $u_1,\ldots,u_{m_B}$ be the rows of $L_2$. As $v_1,\ldots,v_{m_B}$ are linearly independent over $\F_q$ and so are $u_1,\ldots,u_{m_B}$, there must be some $u_i$ that is not in the span of $v_1,\ldots,v_{m_B}$. Hence, $m_B$ is not maximal. 
	For (ii) note that there is at least one equation $L'$ that is induced by $L$  on $B$, by definition of $\mathcal{C}(L)$ and since $B \in \mathcal{C}(L)$. Suppose there are two linearly independent equations induced on $B$. Then some  linear combination of these equations has length strictly less than $s(L)$, a contradiction. 
	Now (iii) follows similarly by supposing there are three linearly independent equations induced on $B$, and finding a linear combination of length  strictly less than $s(L)$. 
\end{proof}	
\begin{proof}[Proof of \thref{lem:useful4}]
	Using elementary row operations, we may assume that the matrix $L$ is such that the final $m_B$ rows correspond to the system $L_B$ and the entry $L_{ij}$ is zero for any $i>m-m_B$ and $j\not\in B.$ 
    Let $y = (y_b)_{b \in B}\in (\F_q)^{|B|}$ and let $L(y)$ be the (not necessarily homogeneous) system of equations obtained from $L=0$ by letting the value of $x_b$ be $y_b$ for each $b \in B$.
	If $y \in \solA{L_B}{\F_q}$, then considering the first $m-m_B$ rows of $L(y)$ gives a (not necessarily homogeneous) system on $k-|B| = k- c(L)$ variables of rank $m-m_B$. Thus, the number of solutions to $L=0$ over $\F_q$ extending $y$ is $q^{k- c(L) - m +m_B}$. The claim for a general $y\in\solA{L_B}{\Fqn}$  follows since each row of such a $y$ (seen as an $(n\times k)$-array) is a solution to $L_B=0$ over $\F_q,$ similar to the argument for \thref{ob:n-vec}.
\end{proof}

\section{Reducing the problem to critical subsystems}\label{sec:main-thm}

	The purpose of this section is to prove Theorem~\ref{t:main-uncommon}, which relates the uncommonness of an $(m \times k)$-system $L$ to the cumulative uncommonness of its critical subsystems. Recall the definition of $\Lambda_L(f)$ from \eqref{eq:lamf}. Say that a function $f: \F_q^n \rightarrow \R$ is \emph{balanced} if $\sum_{x \in \F_q^n}f(x) = 0$.
    
    \begin{thm} \thlabel{t:main-uncommon}
    	Let $q$ be a prime power, let $2\le m < k$ be integers,  and let $L$ be an $(m \times k)$-system over $\F_q$. If there exists $n \ge 1$ and a balanced function $f: \Fqn \rightarrow \fint$ such that 
    	\begin{equation*} 
    		\sum_{B \in \Cc(L)} \Lambda_{L_B} (f) <0,
    	\end{equation*}
    then $L$ is uncommon.
    \end{thm}
As in \thref{lem:useful}, the assumption $m\ge 2$ is merely included to avoid technicalities in the case when $L$ is an odd equation (and the set $\Cc(L)$ would be empty).
\begin{rem}
		Observe that this theorem relates uncommonness of a system to a condition concerning its critical subsystems, which are subsystems of rank at most two, by Lemma~\ref{lem:useful}. So it yields a potential strategy for showing that a high rank system is uncommon, by understanding certain subsystems of low rank. Note also that the condition does not depend on how these subsystems relate to each other within $L$.
\end{rem}

For the proof of Theorem~\ref{t:main-uncommon}, it is convenient to define the following notation. For any $(m \times k)$-system $L$, any $B \subseteq [k]$ and $f: \Fqn \rightarrow \mathbb{R}$, define 
$$\Phi_L(B, f) :=\frac{1}{|\solA{L}{\Fqn}|}\sum_{\bx \in \solA{L}{\Fqn}} \prod_{i \in B} f(x_i).$$
So $\Phi_L([k],f) = \Lambda_L(f)$ and it follows immediately from the definitions of $\com_L$ and $\Phi_L$ that 
\begin{equation}\label{eq:thomasonst}
	\com_L(f) = 2 \sum_{\substack{B \se [k],\\ |B| \text{ even}}} 2^{-k+|B|}  \Phi_L(B, f).
\end{equation} 

Now we prove Theorem~\ref{t:main-uncommon}.
\begin{proof}[Proof of Theorem~\ref{t:main-uncommon}]
	Let $f$ be a function satisfying the hypotheses of the theorem. Define $\alpha = \alpha(f)$ via the equation
	\begin{equation}\label{eq:alpha}
		-2^{k+2}\alpha = \sum_{B \in \Cc(L)} \Lambda_{L_B}(f).
	\end{equation}
As $|\Lambda_{L_B}(f)| \le 1$ for any $B$, by definition and the theorem hypothesis, we have $0 < \alpha \le 1/4$. Let $g = \alpha f$. Note that as $f:\Fqn \rightarrow [-\frac{1}{2},\frac{1}{2}]$, we have $g: \Fqn \rightarrow [-\frac{1}{2},\frac{1}{2}]$. 
To show that $L$ is uncommon it is sufficient to show that $\com_L(g) < 2^{1-k},$ see \thref{def:fctdef}. 
	By \eqref{eq:thomasonst}, we have
\begin{equation}\label{eq:split}
	\com_L(g) = 2\sum_{
	|B| \text{ even}}2^{-k+|B|}  \Phi_L(B, \alpha f) = 2^{1-k} +  2^{1-k}\!\!\!\sum_{
	0 < |B| \text{ even}} (2\alpha)^{|B|}  \Phi_L(B, f), 
\end{equation}
where we omit the condition $B\se[k]$ here and below for clarity. 
It suffices to show that the rightmost sum is negative. We use the first part of the following claim to see that the terms coming from sets with $0<|B|<c(L)$ do not contribute at all. Our choice of $\alpha$ then guarantees that the dominant terms in this sum come from sets in $\mathcal{C}(L)$. This information, together with the second part of the claim and the assumption of the theorem, allows us to deduce the assertion. Call a set $B\se [k]$ {\em rank-reducing} if the matrix obtained from $L$ by removing the columns corresponding to $B$ has rank less than the rank of $L.$

\begin{claim}\label{cl:phi}$ $
    \begin{enumerate} 
		\item~\label{it:generic} If $B$ is not rank-reducing, then $\Phi_L(B, f) = 0$. 
		\item~\label{it:lambda} If $B \in \mathcal{C}(L)$, then $\Phi_L(B, f) = \Lambda_{L_B}(f).$ 
	\end{enumerate}
\end{claim}
\begin{proof}[Proof of Claim~\ref{cl:phi}]
	Let $B \subseteq [k]$ and let $b = (b_i)_{i \in B} \in (\Fqn)^{|B|}$. Let $L(b)$ be the (not necessarily homogeneous) system of equations obtained from $L=0$ by letting the value of $x_i$ be $b_i$ for each $i \in B$ and let $h(b)$ be the number of solutions to $L(b)$. Note that $h(b)$ can be 0. By considering each solution of $L$ as an extension of some $b \in (\Fqn)^{|B|}$, we obtain 
	\begin{equation}\label{eq:phiB}
		\Phi_L(B,f) = q^{n(-k+m)}\!\!\!\!\! \sum_{b \in (\Fqn)^{|B|}}\!\!\!\!\!  h(b) \prod_{i \in B}f(b_i).
	\end{equation}
	
	For~\ref{it:generic}, suppose $B$ is such that the matrix obtained from $L$ by removing the columns corresponding to $B$ has the same rank as $L$. Then for any $b \in (\Fqn)^{|B|}$, we have $h(b) = q^{n(k-|B|-m)}$. So \eqref{eq:phiB} implies
$$\Phi_L(B,f) = q^{n(-k+m)} \cdot  q^{n(k-|B|-m)}  \Big(\sum_{y \in \Fqn} f(y)\Big)^{|B|} =0,$$
since $f$ is balanced, as required for~\ref{it:generic}. 
	
For~\ref{it:lambda}, suppose $B\in \mathcal{C}(L)$ and $L_B$ has rank $t$.  By Lemma~\ref{lem:useful4}, we see that $h(b) = q^{n(k- |B| - m + t)}$ if $b\in \sol(L_B)$ and $h(b) = 0$ otherwise. 
 This and~\eqref{eq:phiB} imply that 
$$\Phi_L(B, f) = q^{n(-k+m)} \!\!\!\!\!\sum_{y \in\sol(L_B)}\!\!\!\!\! q^{n(k- |B| - m + t)} \prod_{i\in B}f(y_i) = q^{n(-|B| + t)} \!\!\!\!\!\sum_{y\in\sol(L_B)} \prod_{i\in B}f(y_i) = \Lambda_{L_B}(f),$$
as required.
\end{proof}

If a set $B$ satisfies $0 < |B| < s(L),$ then  $B$ cannot be rank-reducing, by definition of $s(L).$ Similarly, $B$ is not rank-reducing if $|B|=c(L)$ and $B\not\in \Cc(L),$ by definition of $\mathcal{C}(L).$ By Claim~\ref{cl:phi}\ref{it:generic}, any such set~$B$ satisfies $\Phi_L(B,f)=0$, so the only sets that contribute to the rightmost sum of \eqref{eq:split} are in $\mathcal{C}(L)$ or they satisfy $|B|>c(L)$. Partitioning according to membership of $\mathcal{C}(L)$ gives 
	\begin{equation}\label{eq:split1}
	\sum_{
	0< |B| \text{ even}} (2\alpha)^{|B|}  \Phi_L(B, f) = (2\alpha)^{c(L)} \sum_{B \in \Cc(L)}   \Phi_L(B, f) + \sum_{\substack{  
	|B| \text{ even}\\ |B|> c(L)} }  (2\alpha)^{|B|} \Phi_L(B, f).
	\end{equation}
	As $|\Phi_L(B, f)| \le 1$ and $2\alpha \in (0,1)$, we obtain
	\begin{equation}\label{eq:split2}
		\sum_{\substack{
		|B| \text{ even}\\ |B|> c(L)} } \!\!\!\!\! (2\alpha)^{|B|} \Phi_L(B, f) \le 2^k (2\alpha)^{c(L) + 2} < 2^{k+2}\alpha (2\alpha)^{c(L)} = - (2 \alpha)^{c(L)}\sum_{B \in \Cc(L)}  \Phi_L(B, f),
	\end{equation}
	where in the final equality we combine Claim~\ref{cl:phi}~\ref{it:lambda} with the definition of $\alpha$ (given in \eqref{eq:alpha}) to replace the $2^{k+1}\alpha$ term, and also use the fact that $0 < \alpha \le 1/4$ and every set in $\mathcal{C}(L)$ has even cardinality $c(L)$.
	It immediately follows from \eqref{eq:split1} and \eqref{eq:split2} that the rightmost sum in \eqref{eq:split} is negative, and hence $\com_L(g) < 2^{1-k}$, as required. We can therefore conclude that $L$ is uncommon. 
\end{proof}

We immediately obtain the following consequence of Claim~\ref{cl:phi} together with an equation similar to~\eqref{eq:split}, which is of independent interest and highlights another property of critical systems. Note that we have $c(L)=k$ for a critical $(m\times k)$-system $L.$

\begin{cor}\label{cor:simple} 
		Let $q$ be a prime power, let $L$ be a critical $(m \times k)$-system over $\F_q$ and let $f: \Fqn \rightarrow \mathbb{R}$ be a balanced function. Then
		$$\com_L(f) = 2^{1-k}  + 2\Lambda_{L}(f).$$
\end{cor}

\subsection{Proof of Theorem~\ref{t:main-evens}}\label{subsec:even}

We now give a simple application of Theorem~\ref{t:main-uncommon} to prove Theorem~\ref{t:main-evens}, which says that a system $L$ with $s(L)$ even is uncommon whenever every critical subsystem is. Recall from Lemma~\ref{lem:useful} that when $s(L)$ is even, the critical subsystems are $(1 \times s(L))$-systems.

We require the following lemma about single equation systems. This lemma immediately follows from the arguments given to prove Theorem 1.4 (b) in \cite{fpz19}, but is not stated explicitly there. For the rest of this section (and nowhere else) we use $\E$ to denote the expectation with respect to the probability function given by the statement in \thref{l:random-fcs}.

\begin{lem}
	\thlabel{l:random-fcs}
	Let $q$ be a prime power. Then there exists $\eps>0$ and a probability function $\mathbb{P}$ on the set ${\mathcal F}$ of balanced functions $f: \F_q \to \fint$ such that the following properties hold. 
	\begin{enumerate}
		\item\label{it:evenex} For every uncommon equation $L$ with even support, $\errm{} {\Lambda_L(f)} = 0.$ 
		\item\label{it:posden} There exists a set $\mathcal{F}^+ \subseteq \mathcal{F}$ of positive probability, such that for every $f \in \mathcal{F}^+$ we have $\Lambda_L(f) > \eps$ for every uncommon equation $L$ with even support. 
	\end{enumerate}
\end{lem}

The two statements \ref{it:evenex} and \ref{it:posden} imply \cite[Theorem 1.4 (b)]{fpz19}, namely that, for a given equation $L$ of even length in which the coefficients of $L$ cannot be partitioned into pairs, each summing to zero, there must be a function $f^*:\F_q \to \fint$ satisfying $\Lambda_L(f^*) <0.$ That is, $L$ is uncommon. We stress that, a priori, the function $f^*$ may be different for any given $L$ (even though the function $f$ may be taken to be the same for any collection of even-length equations $L$). 

Now we are ready to prove Theorem~\ref{t:main-evens}.
\begin{proof}[Proof of Theorem~\ref{t:main-evens}.]
	Let $f \in \mathcal{F}$ be chosen randomly according to $\mathbb{P}$. As $s(L)$ is even, by Lemma~\ref{lem:useful}, for each $B \in \Cc(L)$, $L_B$ is an uncommon equation with even support. So by Lemma~\ref{l:random-fcs}~\ref{it:evenex}, we have $\erm{}{\Lambda_{L_B}(f)} = 0$. Thus by linearity of expectation,        \begin{equation}\label{eq:exp0}
		\erm{}{\sum_{B \in \Cc(L)} \Lambda_{L_B}(f)} = 0.
	\end{equation}
	By Lemma~\ref{l:random-fcs}~\ref{it:posden}, there exists a set of positive probability $\mathcal{F}^+ \subseteq \mathcal{F}$ such that for each $f \in \mathcal{F}^+$ we have $ \sum_{B \in \Cc(L)} \Lambda_{L_B}(f) > 0$. As $\mathcal{F}^+$ has positive probability, using \eqref{eq:exp0} we see that with positive probability,
	\begin{equation} \label{eq:fneg}
		\sum_{B \in \Cc(L)} \Lambda_{L_B}(f) < 0.
	\end{equation}	
	 Therefore a function $f^*$ satisfying \eqref{eq:fneg} exists and applying Theorem~\ref{t:main-uncommon} with $f^*$ implies that the system $L$ is uncommon, as required.
\end{proof}

Observing that all irredundant equations with two variables are uncommon gives the following immediate corollary. We will use this in the proof of \thref{t:4APs}.
\begin{cor} \thlabel{c:twos}
	Any irredundant system $L$ with $s(L)=2$ is uncommon.
\end{cor}

\section{Uncommon 4-variable systems}\label{sec:unc-4-var}
   Recall that for a system $L$, we write $s(L)$ to denote the minimal length of an equation induced by $L$. In this section we will prove that any $(2 \times 4)$-system with $s(L) = 3$ is uncommon. We do this by exhibiting a function with the properties required to apply Theorem~\ref{t:main-uncommon}.

    \begin{lem}\thlabel{l:psi-neg} 
   	Let $q$ be a prime power. Then there exists $\nu > 0$, $d \ge 1$ and a balanced function  $\psi: \F_q^d \rightarrow [-\frac{1}{4},\frac{1}{4}]$ such that $\Lambda_{L}(\psi)< -\nu $ for every $(2\times 4)$-system $L$ over $\F_q$ with $s(L)=3$. 
   	%
   \end{lem}
	\thref{l:psi-neg} and Corollary~\ref{cor:simple} together imply the following.
	\begin{cor}\thlabel{c:twoByfour}
		Every $(2 \times 4)$-system $L$ with $s(L)=3$ is uncommon. 
	\end{cor}
   
   \begin{proof}[Proof of \thref{l:psi-neg}]
   	
	Let $q \ge 40$ be a prime power. We first exhibit an appropriate function $\psi$ for such a $q$. We then use this to find a function that works for any prime power $q < 40$. 

	Let $d = (q-3)^2$, set $\alpha := 1/4$ and $\beta = \beta(q,d) := -\alpha (q-1)/d$. 
   	Now	define $\psi = \psi(q): \F_q^d \to [-1/4, 1/4]$ by 
   		\begin{equation}\label{eq:psi}			
   		\psi(y) = \begin{cases}
   			0 & \text{if } y_i = y_j = 0 \text{ for some } i \not= j \in [d],\\
   			\alpha & \text{if } y_i \not= 0 \text{ for all } i \in [d],\\
   			\beta & \text{ otherwise.}
   		\end{cases}
   		\end{equation}
	We remark that we may choose $\alpha\neq 0$ arbitrarily for this proof to work, as long as $|\psi|\le 1/4,$ which we use in the proof of Theorem~\ref{t:4APs}. 
   	By definition, we have
   	\begin{equation} \label{eq:f-zero-sum}
   		\sum_{y \in \F_q^d} \psi(y) = (q-1)^d \alpha + d(q-1)^{d-1}\beta  = 0,
   	\end{equation}
	so $\psi$ is balanced. 
   	
   	We must now check that there exists some $\nu>0$ such that $\Lambda_L(\psi) < -\nu$, for any $(2\times 4)$-system $L$ over $\F_q$ with $s(L)=3.$ 
   	Consider 
   	\begin{align*}
   		\Lambda_L(\psi) &= \frac{1}{|\solA{L}{\F_q^d}|} \sum_{\xv\in\solA{L}{\F_q^d}} \psi(\xv),  
   	\end{align*}
   	where we use $\psi(\xv)$ to denote the product $\psi(x_1)\psi(x_2)\psi(x_3)\psi(x_4)$ for $\xv=(x_1,x_2,x_3,x_4) \in (\F_q^d)^4.$ 
  An element $\xv=(x_1,x_2,x_3,x_4)$ in $\solA{L}{\F_q^d}$ corresponds naturally to a $(d\times 4)$ matrix $M(\xv)$ whose columns are the elements $x_1,x_2,x_3$ and $x_4$. Recall from the beginning of Section~\ref{sec:prelims} that if $\xv$ is a solution to $L$ over $\F_q^d,$ each row of $M(\xv)$ corresponds to a solution to~$L$ in~$\F_q$. 
   	By definition of $\psi$, if some column of $M(\xv)$ contains two zeroes, then $\psi(\xv)= 0$. 
   	As $|\solA{L}{\F_q^d}| = q^{2d},$ it follows that 
   	\begin{align}\label{eq:goal}
   		\Lambda_L(\psi) = q^{-2d} \sum_{\ell=0}^4\sum_{\xv\in Z_\ell} \psi(\xv), 
   	\end{align}
   	where, for $0\le \ell \le 4$, $Z_\ell$ denotes the set of $\xv\in\solA{L}{\F_q^d}$ such that the matrix $M(\xv)$ contains exactly $\ell$ zeros.

	As $L$ is a $(2 \times 4)$-system with $s(L) = 3$, it has the property that for any $a,b \in \F_q$ and $i,j \in [4]$ with $i\neq j,$  there is a unique solution $\xv\in \F_q^4$ to $L$ such that $x_i = a$ and $x_j = b$. 
	It follows that the only solution to $L$ over $\F_q$ with at least two zeroes is $\xv = 0$, and the number of solutions to $L$ in $\F_q$ with exactly one zero is $4(q-1)$ (since a solution can be specified by fixing one zero and one non-zero variable). Hence, the number of solutions with no zeros is $q^2-4q+3 = (q-1)(q-3)$, which we denote by $m_4$. So, for any $\xv \in\bigcup_{\ell=0}^4 Z_\ell$, a non-zero row of $M(\xv)$ contains at most one zero.

   	Define $Z_4^*:= \{\xv \in Z_4: M(\xv) \text{ has a zero row}\}$. Then by definition of $\psi$,
   	\begin{align}\label{eq:ZStar}
   		\sum_{\xv\in Z_4^*}  \psi(\xv) = \sum_{\xv\in Z_4^*} \beta^4 = d m_4^{d-1} \beta^4,
   	\end{align}
 	as there are $d$ choices for the (unique) all-zero row of $M(\xv)$, and $m_4$ choices for the $(d-1)$ remaining rows. 
    For  $0\le \ell \le 4$ and $\xv\in  Z_\ell\sm Z_4^*$ we obtain that either $\psi(\xv) = 0$ (if $M(\xv)$ contains two or more zeros in the same column) or that  
   	$\psi(\xv)= \alpha^{4-\ell}\beta^{\ell}$ (if $M(\xv)$ has $\ell$ zeros in distinct columns). As $\xv\not\in Z_4^*$, the $\ell$ zeros also appear in distinct rows. 
   	There are $\binom 4\ell (d)_{\ell}$ ways to choose the position of the zeros, where $(d)_{\ell}$ denotes the falling factorial. Given the position of the zeros, the number of solutions is $(q-1)^\ell m_4^{d-\ell}$ (by choosing rows of $M(\xv)$ one by one). 
   	We obtain that 
   	\begin{align}\label{eq:Zell}
   		\sum_{\xv\in Z_\ell\sm  Z_4^*} \psi(\xv) = \binom 4\ell (d)_{\ell} (q-1)^{\ell}m_4^{d-\ell}\alpha^{4-\ell}\beta^{\ell}.
   	\end{align}
   	As $Z_0,\ldots, Z_3,Z_4\sm Z_4^*$ and $Z_4^*$ are pairwise disjoint,~\eqref{eq:goal}, \eqref{eq:ZStar} and~\eqref{eq:Zell} imply that 
   	\begin{align}\label{eq:aux003}
   		\Lambda_L(\psi) &= q^{-2d} \left(d m_4^{d-1} \beta^4 
   		+ \sum_{\ell=0}^4 \binom 4\ell (d)_{\ell} (q-1)^{\ell}m_4^{d-\ell}\alpha^{4-\ell}\beta^{\ell}\right)\nonumber\\
   		&= \left(\frac{m_4}{q^{2}}\right)^d \alpha^4 
   		\left( 	\frac{(q-1)^3}{(q-3)d^3} + \sum_{\ell=0}^4 \binom 4\ell (-1)^{4-\ell} \frac{(d)_{\ell}}{d^{\ell}}\left(\frac{q-1}{q-3}\right)^{\ell}
   		\right),
   	\end{align}
   	where we use that $\beta/\alpha = -(q-1)/d$ and that $m_4 = (q-1)(q-3)$. Substituting $d = (q-3)^2$ and rearranging gives that for some $R$ with $|R|\leq 100/ (q-3)^5$, we have
   	\begin{align*}
   		\Lambda_L(\psi) &= \left(\frac{m_4}{q^{2}}\right)^d \alpha^4 
   		\left(	\frac{(q-1)^3}{(q-3)^7} +  \frac{16}{(q-3)^4} - \frac{21}{(q-3)^4}
   		\left(1+\frac{2}{q-3}\right)^2
   		+R
   		\right)\\
   		&\le \left(\frac{m_4}{q^{2}}\right)^d \left(\frac{\alpha}{q-3}\right)^4 
   		\left(	\left(\frac{q-1}{q-3}\right)^3  - 5 + \frac{100}{q-3}	\right)\\
   		&\le - \left(\frac{m_4}{q^{2}}\right)^d \left(\frac{\alpha}{q-3}\right)^4, 
   	\end{align*}
   	where the last inequality holds as $q\ge 40.$ Choosing $\nu = \left((q-1)(q-3)/q^{2}\right)^d \left(\alpha/(q-3\right))^4$ gives a function $\psi$ with the required properties whenever $q\ge 40$.

   	Assume now that $q\le 40$ is a prime power. Then $q^6 > 40$. We have just proved that there exists an integer $d$, some $\nu > 0$ and a balanced function $\psi:\F_{q^6}^{d}\to[-1/4,1/4]$ such that   
   	$\Lambda_{L}(\psi)< -\nu$ for any $(2\times 4)$-system $L$ over $\F_q$ with $s(L)=3.$ 
   	Considering $\F_{q^6}$ as a vector space over $\F_q$, there is a linear isomorphism $h: \F_q^{6d}\to \F_{q^6}^d.$  
   	Now, $\xv=(x_1,\ldots,x_4)\in \F_q^{6d}$ is a solution to $L=0$ in $\F_q^{6d}$ if and only if $(h(x_1),\ldots,h(x_4))$ is a solution to $L=0$ in $\F_{q^6}^d,$ since $L$ is a system of linear forms with coefficients in $\F_q$ and since $h$ is a linear isomorphism. It follows that $\psi\circ h:\F_q^{6d}\to[-1/4,1/4]$ is a function with the required properties.
   	\end{proof}

\section{Proof of Theorem~\ref{t:4APs}}\label{sec:ProofOfMain1} 

In this section we draw together results from Sections~\ref{sec:main-thm} and~\ref{sec:unc-4-var} along with two new lemmas to prove Theorem~\ref{t:4APs}.

\subsection{Overview of proof}

We begin by motivating our two key lemmas, from which the proof will easily follow. Let $q$ be a prime power, let $m < k$ and let $L$ be an  $(m \times k)$-system over $\F_q$. In order to prove \thref{t:4APs}, we wish to apply Theorem~\ref{t:main-uncommon}. In order to do this, for our system $L$ we must construct a balanced function $\psi^*$ that has the property
\begin{equation}\label{eq:rough}
	\sum_{B \in \mathcal{C}(L)} \Lambda_{L_B}(\psi^*) < 0.
\end{equation}	
	 Given \thref{c:twos}, the only case needing work is that where $s(L) = 3$. Using Lemma~\ref{lem:useful}, we see that the systems $L_B$ of interest in \eqref{eq:rough} have rank either one or two, and our two key lemmas deal with these cases respectively. Our proofs for both lemmas were inspired by arguments for analogous results given in \cite{gowers20} and in~\cite{sw17}.

We will define an operator $G = G_n^{\alpha}$ that converts any function to one with small Fourier coefficients. Our choice of $\psi^*$ will be a translate of $G[\psi]$, where $\psi$ is the function defined in \eqref{eq:psi}. The property of having small Fourier coefficients used with  \eqref{eq:single-eq} will allow us to show that, for any system $L_B$ consisting of a single equation, $\Lambda_{L_B}(\psi^*)$ is vanishingly small as $n \rightarrow \infty$ and hence is negligible in our consideration of the sum in \eqref{eq:rough}. 

In order to define the operator, we will first introduce some notation.
Let $q$ be a prime power, say $q=p^{\kappa}$ for a prime $p$ and a non-negative integer $\kappa$, and let $d\le n$ be  non-negative integers. For a function $f: \F_q^d \to \R$ 
define $\fdag: \Fqn \to \R$ by 
        \begin{equation} \label{eq:f-ext}
            \fdag (x_1, \dots, x_n) = f(x_1, \dots, x_d).
        \end{equation}%
Now, let $t$ be an integer and let $\alpha = (\alpha_1, \dots, \alpha_t) \in \F_q^t.$ 
For a function $f: \F_q^d \to \R$ 
define the function $\gownn{f}: \Fqn \to \R$ by 
\begin{align}\label{def:gowers}
	\gownn{f}(x) = \frac{1}{2t} \fdag(x) \sum_{j=1}^t  \big( \omega^{\Tr(\alpha_j ( x \cdot x ))} + \omega^{- \Tr(\alpha_j (x \cdot x ))} \big),
\end{align}
where $\omega =\exp(2\pi i/p),$ and $\Tr:\F_q\to\F_p$ is the standard trace map. 
We note at this point that for every $f: \F_q^d \to \R$ and every $x\in\F_q^n,$ we have $| \gownn{f}(x) | \le \max_{y}| f(y) |,$ where the maximum is over $y\in\F_q^d.$ 

Our first key lemma shows that by applying the operator $\gownn{f}$ to any given function $f:\F_q^d \to [-1/2,1/2],$ the Fourier coefficients can be made arbitrarily small by choosing $n$ large enough. The crucial property of the operator that allows this bound is that the function $f^{\dagger}$ depends only on the first $d$ coordinates of its argument.

\begin{lem} \thlabel{l:fhat-zero}
	Let $q$ be an odd prime power, let $d\le n$ and $t$ be integers, let $f:\F_q^d \to [-1/2,1/2]$ be a function, and let $\alpha\in (\F_q^\times)^t$ be given.  
	Then for all $r \in \widehat{\F_q^n}$, we have $|\widehat{\gownn{f}}(r)| \le  q^{d-n/2}.$
\end{lem}

We postpone the proof to Subsection~\ref{subs:Fourier}. An analogous result was proved in~\cite{gowers20} and in~\cite{sw17} for the special case when $L$ describes a 4-AP and $\F_q^d$ is either $\Z_n$ or $\F_5.$ In fact, we only apply the lemma when $t=4$, but we include the general statement here in anticipation of wider applicability. 

As mentioned above, \thref{l:fhat-zero} allows us to show that the contribution to the sum in \eqref{eq:rough} from critical subsystems of rank one is negligible. So we turn our attention to the consideration of the rank-two critical subsystems. Our second key lemma shows that we are able to choose some $\alpha$ depending on a particular rank-two critical subsystem $L^*$ satisfying $s(L^*) = 3$, such that  $\Lambda_{L^*}(\psi^*)$ contributes a dominant negative term to the sum in \eqref{eq:rough} (and the other critical subsystems either contribute negligibly or also provide a large negative term). Hence our key lemmas together ensure that \eqref{eq:rough} is satisfied.

\begin{lem} \thlabel{l:gowers}
	Let $q$ be an odd prime power and let $L^*$ be a $(2 \times 4)$-system over $\F_q$ with $s(L^*)=3.$ Let $d\le n$ be non-negative integers, let $f: \F_q^d \to \left[-\frac 12, \frac 12 \right]$ be a  function, and let $\fdag: \Fqn \to \R$ be defined as above. Then there exists an $\alpha \in (\F_q^{\times})^4$ such that the following holds. For every $(2 \times 4)$-system $L$ over $\F_q$ there is a non-negative integer $K_L$ such that 
	$$ \Big|\Lambda_L(\gownn{f}) - 2^{-12} K_L \Lambda_L (f)\Big| \leq 16\; q^{2d-n/2}.$$ 
	Moreover, $K_{L^*}\ge 1.$
\end{lem}        

The proof of \thref{l:gowers} is included in Subsection~\ref{subs:Fourier}.

\subsection{The main proof} 
We now apply our two key lemmas (\thref{l:fhat-zero} and \thref{l:gowers}) to prove \thref{t:4APs}. 
\begin{proof}[Proof of \thref{t:4APs}]
	Let $q$ be an odd prime power, and let $2 \le m < k$ be integers. Let $L^*$ be a $(2 \times 4)$-system and let $L$ be an irredundant $(m\times k)$-system inducing $L^*,$ say on the subset $D\se [k]$. Note that this necessarily implies that $L^{\ast}$ is irredundant as well. 
	Since $L^*$ is a $(2\times 4)$-system induced by $L$, it follows that $s(L) \le 3$. 
	If $s(L) = 1$, then $L$ induces an equation $x_i = 0$, for some $i \in [k]$. Considering the set $A = \{0\}$ and taking $n$ sufficiently large gives $$|\sol(L;A)| + |\sol(L;\compl{A})| = 1 <  2^{1-k}q^{n(k-m)}= 2^{1-k}|\sol(L;\Fqn)|,$$ where the final equality is simply  \thref{ob:n-vec}. Hence, $L$ must be uncommon in this case, by~\eqref{eq:commonSet}. 
    The case $s(L) = 2$ directly follows from \thref{c:twos}.  
	
	So suppose $s(L)=3.$ 
	Let $d,\nu>0$ and a balanced function $\psi:\F_q^d\to[-1/4,1/4]$ be given by \thref{l:psi-neg} so that 
	\begin{align}\label{eq:aux431}
		\Lambda_{L'}(\psi) < -\nu 
	\end{align}
	for any $(2\times 4)$-system $L'$ with $s(L')=3.$ 
	Let $n\ge 8d$ be a large enough integer and 
	let $\alpha \in (\F_q^\times)^4$ be given by \thref{l:gowers} (applied with $f=\psi$) with the property that
	\begin{align}\label{eq:aux885}
		\Big|\Lambda_{L'}(\gownn{\psi}) - 2^{-12} K_{L'} \Lambda_{L'} (\psi)\Big| \leq 16\; q^{2d-n/2} \le 16q^{-n/4}
	\end{align}
	for every $(2\times4)$-system $L'$, where $K_{L'}$ is a non-negative integer, and $K_{L^*}\ge 1.$ 
	Let $G_n=\gownn{\psi},$ 
	let $\mu = \er{G_n} = \widehat{G_n}(0),$ 
	and note that 
	\begin{align}\label{eq:aux801}
	|\mu|\le q^{d-n/2}\le q^{-n/4}\le 1/4,
	\end{align}
	by \thref{l:fhat-zero} and assumption on $n.$ 	
	We define $\psi^*: \Fqn \to \left[- \frac 12, \frac 12 \right]$ by 
	$\psi^*(x) = G_n(x) - \mu$,
	which gives $\sum_{x \in \Fqn} \psi^*(x) = 0$, 
	and remark that, indeed, 
	$$|\psi^*(x)|\le |G_n(x)|+|\mu|\le \textstyle\max_y |\psi(y)|+|\mu|\le 1/4+1/4 = 1/2$$
	for all $x\in \F_q^n.$ 
	We will show that
	\begin{equation} \label{eq:sl3}
		\sum_{B \in \Cc(L)} \Lambda_{L_B}(\psi^*) < - 2^{-12}\nu + O(q^{-n/4}),
	\end{equation}
	where the implicit constant in big-$O$ does not depend on $n.$  
	By Lemma~\ref{lem:useful}, for each $B \in \Cc(L)$, the system $L_B$ is a $(t \times 4)$-system, for $t \in \{1,2\}$. For $i=1, 2,$ let $\Cc_i(L)$ be the set of $B\in\Cc(L)$ such that $L_B$ has rank $i$. 
	
	Consider first $B \in \Cc_1(L).$ 
	Then by~\eqref{eq:single-eq},
	\begin{align*} 
		\Lambda_{L_B}(\psi^*)= \sum_{r \in \widehat{\mathbb{F}^n_q}} \widehat{\psi^*}(b_1 r) \cdots \widehat{\psi^*}(b_4 r),
	\end{align*} 
where $b_1,b_2,b_3,b_4\in\F_q$ are the coefficients of (the single equation) $L_B.$ 
	Now, $\widehat{\psi^*}(0) = \er{\psi^*} = 0$ and, for $r\neq 0,$ we have that 
	$|\widehat{\psi^*}(r)| = |\widehat{G_n}(r)| \le q^{d-n/2},$ by \thref{l:fhat-zero}.	It follows that 
	\begin{align*} 
		\big|\Lambda_{L_B}(\psi^*)\big| \le q^{n+4(d-n/2)} =  q^{4d-n} \le q^{-n/2}
	\end{align*} 
	for any such $B$. 

	Now consider $B \in \Cc_2(L).$ Then 

	\begin{align*}
		\Lambda_{L_B}(\psi^*) =  \frac{1}{|\sol(L_B;\mathbb{F}_q^n)|}\sum_{\xv \in  \sol(L_B;\mathbb{F}_q^n)} \sum_{S \subseteq [4]}(- \mu)^{4-|S|}\prod_{j \in S}G_n(x_j).\\
		\end{align*}
	The term given by $S = [4]$ is $\Lambda_{L_B}(G_n)$ and hence 
	\begin{align}\label{eq:aux059}
	\big|\Lambda_{L_B}(\psi^*) - \Lambda_{L_B}(G_n) \big| \le 15|\mu|,
	\end{align} 
	as $|\mu| < 1 $ and $|G_n(x)| < 1$ for any $x \in \mathbb{F}_q^n$, by \eqref{def:gowers}. 
	Now, the bound $|\mu|\le q^{-n/4}$ from \eqref{eq:aux801}, 
	together with~\eqref{eq:aux059} and~\eqref{eq:aux885} applied with $L' = L_B$ implies that 
	\begin{align}\label{eq:aux055}
		\big|\Lambda_{L_B}(\psi^*) - 2^{-12} K_{L_B} \Lambda_{L_B} (\psi)\big| \le 2^5q^{-n/4}.
	\end{align}
	In particular, using~\eqref{eq:aux431} to bound each $\Lambda_{L_B}(\psi),$ and since there are at most $\binom{k}{4}$ sets $B$, we have 
		\begin{align*}
			\sum_{B \in \mathcal{C}(L)}\Lambda_{L_B}(\psi^*) \le -2^{-12}\nu\sum_{B \in \mathcal{C}_2(L)}K_{L_B} + (2k)^4 q^{-n/4}.
		\end{align*}
		
	Now for some $B \in \mathcal{C}_2(L)$ we have that $L_B$ is equivalent to $L^*$ since $L$ induces $L^*.$ Thus $K_{L_B} = K_{L^*} \ge 1$ for some $B\in \mathcal{C}_2(L).$ Since also $K_{L_B}$ is non-negative for every $B \in \mathcal{C}_2(L)$, we obtain that $\sum_{B \in \mathcal{C}_2(L)}K_{L_B} \ge 1$, which implies that 
	$$\sum_{B \in \Cc(L)} \Lambda_{L_B}(\psi^*) < -2^{-12}\nu + (2k)^4 q^{-n/4} \le - 2^{-13}\nu,$$ 
	for sufficiently large $n$. Theorem~\ref{t:main-uncommon} implies that the system $L$ is uncommon.
\end{proof}

\subsection{Proofs of \thref{l:fhat-zero,l:gowers}}\label{subs:Fourier}

We first prove an auxiliary lemma that will be used in the proofs of both our key lemmas. Given a prime power $q$, say $q=p^\kappa$ where $p$ is prime, recall the trace map $\Tr:\F_q\to \F_p$ which is linear as a map between vector spaces over $\F_p$. Apart from linearity of $\Tr$, we use that it is non-degenerate, that is 
$\Tr(xy)=0$ for all $x\in\F_q$ if and only if $y=0.$  For $d \in \N$, write $0^d$ to denote the $d$-dimensional zero vector. 

\begin{lem}\thlabel{lem:Fourier-calculation}
	Let $q$ be an odd prime power, let $n\in\N,$ let $W:= 0^d \times \F_q^{n-d}$ and for $i=1,2$ and $x_i\in\F_q^d,$ let $W_i=(x_i,0^{n-d})+W$ be an affine subspace. For $\alpha,\beta,\gamma\in \F_q$, let $Q(x,y)= \alpha ( x \cdot x) + \beta ( x\cdot y) +\gamma ( y\cdot y)$ be a quadratic form on $\F_q^n\times \F_q^n$.  Then 
	\begin{enumerate}
		\item \label{aux:item1} if $\alpha\neq 0$,  for every $z\in \F_q^n$ we have $\big|\E_{x\in \F_q^n} \ind_{W_1}(x)\omega^{\Tr(Q(x,z))} \big| 
		\le q^{-n/2}$ ; and 
		\item \label{aux:item2} if $Q(x,y)\not\equiv 0$, we have  $\big|\E_{x,y\in \F_q^n} \ind_{W_{1}}(x)\ind_{W_2}(y) \omega^{\Tr(Q(x,y))} \big| \le q^{-n/2}.$  
	\end{enumerate}
\end{lem}

\begin{proof}
	Denote by $\widetilde{x}_1$ the vector $(x_1,0^{n-d}).$ For every $z\in\F_q^n,$ as $ \widetilde{x}_1\cdot w = 0$  for every $w\in W$ and using the linearity of the trace function, we have 
	\begin{align}\label{eq:shift-affine-linear}
		\Big|\sum_{x \in W_1}\omega^{\Tr(Q(x,z))} \Big| 
	 =\Big|\sum_{w \in W}\omega^{\Tr(Q(\widetilde{x}_1+w,z))} \Big| 
	 = \Big|\sum_{w \in W}\omega^{\Tr(\alpha ( w\cdot w) + \beta ( w\cdot z))} \Big|.  
	\end{align}
	We claim that if $z=(z_1,\ldots,z_n),$ then
	\begin{align}\label{eq:basic-sum}
		\sum_{w \in W}\omega^{\Tr(  w\cdot z)} = 
		\begin{cases}
			|W| & \text{ if } z_{d+1}=\ldots=z_n=0,\\
			0 & \text{ otherwise.}
		\end{cases}
	\end{align}
 To see this, first suppose $z_{d+1}=\ldots=z_n=0$. Then for any $w \in W$, we have $\Tr(w\cdot z) = \Tr(0) =  0.$ For the other case, suppose that for some $i \in \{d+1,\ldots, n\}$, we have $z_i \not= 0$. Then, by the trace map properties, there exists some $\lambda \in \F_q$, such that $\Tr(\lambda z_i) \not=0$. Therefore, letting ${\bf e}_i$ be the $i$-th standard basis vector, we have $y := \lambda {\bf e}_i \in W$ satisfies $\Tr( y\cdot z) \not= 0$ and 
	$$\sum_{w \in W}\omega^{\Tr( w\cdot z)} = \sum_{w \in W}\omega^{\Tr(( w+y)\cdot  z)} = \omega^{\Tr( y\cdot z)}\sum_{w \in W}\omega^{\Tr( w\cdot z)}.$$
	Hence the left hand side of \eqref{eq:basic-sum} is zero in this case. 
	
Now if $\alpha\neq 0$ then using linearity of the trace map we obtain that 	
	\begin{align*} 
	\Big|\sum_{w \in W}\omega^{\Tr(\alpha (w\cdot w) + \beta (w\cdot z))} \Big|^2
		&=  \Big|\sum_{v,w \in W}\omega^{\Tr(\alpha ( v\cdot v  - w\cdot w) + \beta ( z\cdot (v-w)))} \Big|\\
		&= \Big|\sum_{u,v \in W}\omega^{\Tr ((2\alpha v-\alpha  u+ \beta z)\cdot u)} \Big| \nonumber\\ 
		&\le \sum_{u \in W}\Big| \omega^{\Tr((\beta z - \alpha u)\cdot u )} \Big| \sum_{v \in W} \omega^{\Tr(2\alpha ( v\cdot u))} \Big| \\
		&\le \sum_{u \in W} \Big| \sum_{v \in W} \omega^{\Tr(2\alpha ( v\cdot u))} \Big| 
		\\&=  |W| \le q^n,
	\end{align*} 
	where in the second-to-last step we use~\eqref{eq:basic-sum} and that $2\alpha\neq 0.$ 
	This together with~\eqref{eq:shift-affine-linear} implies that 
	$$\big|\E_{x\in \F_q^n} \ind_{W_1}(x)\omega^{\Tr(Q(x,z))} \big| 
	= q^{-n} \; \big|\sum_{x\in W_1} \omega^{\Tr(Q(x,z))} \big| \le q^{-n/2},$$
	which proves~\ref{aux:item1}. 
	For~\ref{aux:item2}, we see that when $\alpha\neq 0$ the claim follows immediately from~\ref{aux:item1} as 
	$$R:= \big|\E_{x,y\in \F_q^n} \ind_{W_{1}}(x)\ind_{W_2}(y) \omega^{\Tr(Q(x,y))} \big| 
	\le \E_{y\in \F_q^n} \big|\E_{x\in \F_q^n} \ind_{W_1}(x)\omega^{\Tr(Q(x,y))} \big|. 
	$$
	Similarly, the claim follows if $\gamma\neq 0$ by swapping $x$ and $y$. Finally, if $\alpha=\gamma=0$ then $\beta\neq 0$ (as we assume $Q\not\equiv 0$). In this case, 
	\begin{align*}
		R \le q^{-2n}\; \sum_{y \in \F_q^n} \big| \sum_{x \in W_1} \omega^{\Tr(Q(x,y))} \big| 
		= q^{-2n}\; \sum_{y \in \F_q^n} \big| \sum_{w \in W} \omega^{\Tr(\beta ( w\cdot y))} \big|
		= q^{-2n}q^{d} |W| =q^{-n},
	\end{align*}
	by~\eqref{eq:shift-affine-linear} and~\eqref{eq:basic-sum}, respectively. 
\end{proof}

Lemma~\ref{l:fhat-zero} follows from~\thref{lem:Fourier-calculation}~\ref{aux:item1} simply by noticing that $\fdag$ is a linear combination of $q^d$ indicator functions of affine subspaces of $\Fqn$. 

\begin{proof}[Proof of \thref{l:fhat-zero}]
	Given $f$ and $n$, recalling the definition of $\fdag : \F_q^n \to [-1,1]$ from~\eqref{eq:f-ext}, we can write  
	$$\fdag(x) = \sum_{z \in\F_q^d} f(z) \ind_{W_z}(x),$$ where for $z \in \F_q^d$, we have that $W_{z}$ is the affine subspace 
	$(z,0^{n-d}) + W,$ 
	and $W= 0^d \times \F_q^{n-d}$. Then the definition of $\gownn{f}$ in~\eqref{def:gowers} gives
	$$\gownn{f}(x) = \frac{1}{2t} \sum_{j=1}^t \sum_{z \in \F_q^d} 
	f(z) \ind_{W_z}(x) \left( \omega^{\Tr(\alpha_j ( x\cdot x ))} +\omega^{- \Tr(\alpha_j ( x\cdot x ))} \right),$$ 
	for $x\in\F_q^n.$ 
	For $\beta \in\F_q$ and $z\in \F_q^d,$ let $H_{\beta,z}$ be the function defined by 
	$H_{\beta,z}(x) = \ind_{W_z}(x) \; \omega^{\Tr(\beta (x\cdot x) )}.$  
	Then for all $\beta\in\F_q^{\times}$ and all $\chi_r\in \widehat{\F_q^n},$ 
	$$|\widehat{H_{\beta,z}}(\chi_r) | 
	= \big| \E_{x\in\F_q^n} \ind_{W_{\tilde x}}(x) \; \omega^{\Tr(\beta ( x\cdot x)) -\Tr(r\cdot x)}  \big| \le q^{-n/2},$$
	by linearity of the trace function and \thref{lem:Fourier-calculation}\ref{aux:item1}. It follows that 
	
	$$|\widehat{\gownn{f}}(r)| \le \frac{1}{2t} \sum_{j=1}^t 
	\sum_{z\in \F_q^d} 
	| f(z)| \;  \Big( \big| \widehat{H_{\alpha_j,z}}(r)\big| +\big|\widehat{H_{-\alpha_j,z}}(r)\big| \Big) 
	\le q^{d-n/2},$$
	where we use  linearity of taking the Fourier transform, that $\alpha_j\neq 0$ for all $j\in[t]$ and that $|f(z)|\le 1$ for all $z\in \F_q^d.$
\end{proof}

\begin{proof}[Proof of \thref{l:gowers}]
	Any $(2\times 4)$-system has a 2-dimensional solution space. 
	Up to reordering the variables we may assume that the solution set of $L^*(\xv)=0$ can be parametrised by
	$$\sol(L^*; \Fqn)=\big\{(x, y, a_1^*x + a_2^*y, b_1^*x + b_2^*y): x, y \in \Fqn\big\},$$ 
	 where $a_1^*,a_2^*,b_1^*,b_2^* \in \F_q.$  
	That is, $L^*$ is equivalent to the system 
	$\begin{psmallmatrix}
		a_1^*& a_2^*&-1& 0 \\ b_1^*& b_2^*&0&-1
	\end{psmallmatrix}.$ 
	Now $s(L^*) = 3$ means that no equation induced by $L^*$ has length less than 3, which implies in particular that 
	\begin{align}\label{eq:coeff-conditions} a_1^*,a_2^*,b_1^*,b_2^* \neq 0 \text{ and }  a_1^*b_2^*-a_2^*b_1^* \neq 0.\end{align}
	We define 
    \begin{align*} 
		\alpha_1 &=a_2^*(b_1^*)^{-1}(a_1^*b_2^*-a_2^*b_1^*), & \alpha_2 &=(a_1^*)^{-1}b_2^*(a_2^*b_1^*-a_1^*b_2^*),  \\
		\alpha_3 &= - a_2^*b_2^*(a_1^*b_1^*)^{-1}, &\text{ and \quad  } \alpha_4 &= 1, 
	\end{align*}
	and notice that the conditions~\eqref{eq:coeff-conditions} imply immediately that $\alpha_1,\ldots,\alpha_4\neq 0.$ Let $\alpha=(\alpha_1,\ldots,\alpha_4).$ 
		Now consider the quadratic form  $Q^*$ on $\F_q^n\times\F_q^n$ defined by 
	$$Q^*(x,y) = \alpha_1|x|^2 + \alpha_2 |y|^2 + \alpha_3 |a_1^*x + a_2^* y|^2 + |b_1^*x + b_2^*y|^2,$$
	where we write $|v|^2$  for $ v \cdot v$. The choice of $\alpha_1,\ldots,\alpha_4$ above implies that $Q^* \equiv 0$ on $\F_q^n\times \F_q^n$ (this can be seen by considering the coefficients of $|x|^2$, $ x\cdot  y$ and of $|y|^2,$ respectively).

	Now, let $L$ be an arbitrary $(2\times 4)$-system over $\F_q$ 
	and let 
	\begin{equation}\label{eq:sol}
		\sol(L; \Fqn)=\big\{(x, y, a_1 x + a_2y, b_1x + b_2y): x, y \in \Fqn\big\}
	\end{equation}
	be a parameterisation of the solution set of $L(\xv)=0$ (again, after possible reordering of the variables), where $a_1,a_2,b_1,b_2 \in \F_q.$  

	Let $\mathcal{Q}_L$ be the set of quadratic forms $Q: (\Fqn)^2 \to \mathbb{C}$ of the form
	$$Q(x, y)= \beta_1|x|^2 + \beta_2 |y|^2 + \beta_3 |a_1x + a_2 y|^2 + \beta_4 |b_1x + b_2y|^2,$$
	where $\beta_j \in \{\pm \alpha_1, \pm \alpha_2, \pm \alpha_3, \pm \alpha_4 \}$, with repetitions allowed. Write $G_n$ for $\gownn{f}$ and define 
	$h(x,y) = 2^{-12}\fdag(x)\fdag(y)\fdag(a_1 x + a_2 y)\fdag(b_1 x + b_2 y).$ 
	Then by \eqref{def:gowers} and \eqref{eq:sol} we have 
	\begin{align}\label{eq:gown-start}
		\Lambda_L(G_n) 
		&= \E_{(x,y) \in (\mathbb{F}_q^{n})^2}G_n(x)G_n(y)G_n( a_1 x + a_2 y )G_n( b_1 x + b_2 y )\nonumber\\
		&= \sum_{Q \in \cQ_L} \E_{x,y}h(x,y)\omega^{\Tr(Q(x,y))},
	\end{align}
    where in the second equality we use linearity of the trace function.
    If $\Tr\circ Q\equiv 0$ on $\F_q^n\times \F_q^n$ (which can be shown to be the case only if $Q\equiv 0$) then the corresponding term in the sum is 
    \begin{align}\label{eq:f-ext1}
		\E_{x,y} h(x,y) &= 2^{-12}\Lambda_L(\fdag) 
		=  2^{-12} \Lambda_L(f),
	\end{align}
    where the identity $\Lambda_L(\fdag) 
		=   \Lambda_L(f)$ can be justified as follows. Since  $\xv\in\solA{L}{\F_q^n}$ if and only if  every row of the solution array $(x_1, x_2 , x_3 ,x_{4})$ is a solution to $L$ in $\F_q$, we have that
	\begin{align*}
		\frac{1}{ |\solA{L}{\F_q^n}|} \sum_{\xv\in\solA{L}{\F_q^n}} \prod_{i=1}^{4} \fdag(x_i) 
		= \frac{1}{ |\solA{L}{\F_q}|^n} \sum_{\xv\in\solA{L}{\F_q^d}} \prod_{i=1}^{4} f(x_i) \cdot |\solA{L}{\F_q}|^{n-d}.
	\end{align*}
	Thus, if we let $\cQ_L^\times$ denote the set of quadratic forms in $\cQ_L$ 
	such that $\Tr\circ Q\not\equiv 0$  and $K_L := |\cQ_L\sm\cQ_L^\times|,$   then~\eqref{eq:gown-start} and~\eqref{eq:f-ext1} imply that 
	\begin{align}\label{eq:aux593}\Lambda_L(G_n) = 2^{-12}K_L\Lambda_L(f) + \sum_{Q\in \cQ_L^\times} \E_{x,y} h(x,y)\omega^{\Tr(Q(x,y))}.\end{align}
	Now let $Q \in\cQ_L^\times.$ 
	Since $\fdag(x)$ only depends on the first $d$ coordinates of $x$ we may write 
	$$h(x,y) = 2^{-12}\sum_{u \in\F_q^d}\sum_{z \in\F_q^d} 
	\ind_{W_{u }}(x)\ind_{W_{z}}(y) c_{u,z},$$ 
	where $W_{u}$ is the affine subspace 
	$(u,0^{n-d}) + W,$ $W:= 0^d \times \mathbb{F}_q^{n-d}$
	and 
	$ c_{u,z} = f(u)f(z) f(a_1u+a_2z) f(b_1u+b_2z).$
	Thus, 
	\begin{align*}\label{eq:f-ext3}
		\left|\E_{x,y} h(x,y)\omega^{\Tr(Q(x,y))} \right| &\le 2^{-12}\sum_{u \in\F_q^d}\sum_{z \in\F_q^d} |c_{u,z}|  \left|\E_{x,y} \ind_{W_{u}}(x)\ind_{W_{z}}(y) \omega^{\Tr(Q(x,y))} \right| 
		\le q^{2d-n/2},
	\end{align*}
	by \thref{lem:Fourier-calculation} since $\Tr\circ Q\not\equiv 0$ (and thus $Q\not\equiv 0$), and since each $|c_{u,z}| \le 1.$ 
	The assertion of the lemma for $L$ now follows  from~\eqref{eq:aux593} as $K_L$ is clearly a non-negative integer and $|\cQ_L|=16.$ 
	
	Finally, we note that $K_{L^*}\ge 1$ as $\alpha_1,\ldots,\alpha_4$ are chosen such that the form $Q^*$ above vanishes everywhere and thus $\Tr\circ Q^*\equiv 0$. 
\end{proof}

\section{Concluding remarks}\label{sec:rems}
In this paper, we resolve a question of Saad and Wolf by showing that any irredundant system $L$ over $\F_q$ inducing a 4-AP is uncommon. In fact, \thref{t:4APs} is much more general as it applies to any irredundant  $(2 \times 4)$-system, not just a 4-AP. An important step in the proof is to show that $(2\times 4)$-systems with $s(L)=3$ are uncommon, (see ~\thref{l:psi-neg}). We do not know whether there is an analogue of this lemma for larger $k$, but conjecture the following.
\begin{conj}\thlabel{conj:2Xk}
   For even $k \ge 6$ and large odd $q$, any $(2 \times k)$-system $L$ with $s(L) = k-1$ is uncommon. 
\end{conj}
For odd $k,$ the situation seems to be more delicate as we have found examples of $(2 \times 5)$-systems with $s(L) = 4$ that are common, and examples with the same parameters that are uncommon. We give a more thorough discussion in~\cite{klm21-2}.

We remark that even if a function $\psi=\psi(k,q)$ as in \thref{l:psi-neg} is found, certifying \thref{conj:2Xk}, we currently also lack an appropriate  generalisation of \thref{l:gowers} to $k>4$ to prove an analogue of \thref{t:4APs} for this setting. 

\medskip 

\paragraph{{\bf Further applications of \thref{t:main-uncommon}}} 

An exciting consequence of \thref{t:main-uncommon} is that it opens up new avenues for studying commonness of linear \emph{systems} by utilising tools from discrete Fourier analysis. These methods were crucial in~\cite{fpz19} to find the characterisation of common 1-equation systems via the identity~\eqref{eq:single-eq}. Yet, for two or more equations, we are not aware of such a direct approach. 

Our application of such techniques (as in ~\eqref{def:gowers}) was inspired by the seminal work of Gowers~\cite{gowers17}; we remark that similar ideas also appear in ~\cite{fpz19,gowers20,sw17}. Although we do use discrete Fourier analysis to show that the uncommonness of the $(2\times 4)$-subsystem can be `transferred' to show uncommonness of the ambient system $L,$ (see~\thref{l:fhat-zero,l:gowers}), these techniques are complemented by other, more direct methods. One particular example of this is the function in \thref{l:psi-neg} which shows that certain $(2\times 4)$-systems are uncommon. This function is explicitly constructed (without the use of discrete Fourier techniques). It would be interesting to see how far the utilisation of Fourier methods can go towards characterising systems of equations, though we believe that to harness the full power of these techniques, other ideas are needed.

We demonstrate the power of using \thref{t:main-uncommon} together with the Fourier approach for 1-equation systems by proving   \thref{t:main-evens} which asserts that if $s(L)$ is even and all the critical equations are uncommon, then $L$ is uncommon. It would be interesting to know whether this condition can be relaxed to requiring just the majority of the critical equations to be uncommon. 
\begin{question}
    Let $L$ be a system such that more than half of the critical systems $L_B,$ for $B\in \Cc(L),$ are uncommon. Is it true that $L$ must be uncommon as well? 
\end{question}

When $s(L)$ is even, a construction using random Fourier coefficients similar to what is used in~\cite{fpz19} and in Theorem~\ref{t:main-evens} may certify uncommonness. 

With the aim of finding further applications of \thref{t:main-uncommon}, we ask the following. 
\begin{question}
    \label{conj:halfdensity}
 For a $k$-variable system $L$, suppose that there is a set $A\se \Fqn$ such that the density of monochromatic solutions in $(A,\compl{A})$ is less than $\alpha^k + (1-\alpha)^k,$ where $\alpha= |A|/q^n$. Must there exist a set $A'\se\F_q^{n'}$ (for some $n'\ge 1$) of  density roughly 1/2 such that the density of monochromatic  solutions in $(A',\compl{A'})$ is less  than $2^{1-k}$?
\end{question}
An affirmative answer would show that it suffices to restrict our attention to sets of density $1/2$ (which correspond to balanced functions). This would yield a partial converse of Theorem~\ref{t:main-uncommon}, possibly allowing us to find more common systems.
\medskip 

\paragraph{{\bf Sidorenko systems.}}
A system is {\em Sidorenko} if for any $n\ge 1$ and any $A\se \Fqn$ the number of solutions to $L=0$ in $A$ is asymptotically (as $n\to\infty$)  at least the expected number of solutions in a random set of density $|A|/q^n$. It is easy to see that a system is common if it is Sidorenko. Linear homogeneous equations that are Sidorenko are fully characterised~\cite{fpz19,sw17}. In a companion paper~\cite{klm21-2}, we analyse sufficient and necessary conditions for a system of two or more equations to be Sidorenko. One of our main results states that any system $L$ with $s(L)$ odd is not Sidorenko. 
In the other direction, we also find a large family of systems that \emph{are} Sidorenko. These system are formed by combining various Sidorenko equations in a block-like fashion. For a precise statement, we refer the interested reader to~\cite{klm21-2}, where we also gather some related open questions.

\medskip

\section*{Acknowledgements} Partially supported by the Australian research council (DP180103684) and by the European Union’s Horizon 2020 research and innovation programme [MSCA GA No 101038085].
Part of this research was carried out during the Graph Theory Downunder workshop at the mathematical research institute Matrix. We would like to thank Matrix for its support and hospitality. 
We would also like to thank the anonymous referee for their careful reading and very helpful comments.

   \bibliographystyle{abbrv} 

\end{document}